\documentclass{article}
\usepackage{amsmath, amssymb, amsthm, epsf, graphicx, amscd, amsfonts, hyperref}

\usepackage{tikz}

\newtheorem{thm}{Theorem}[section]
\newtheorem{cor}[thm]{Corollary}
\newtheorem{lem}[thm]{Lemma}
\newtheorem{prop}[thm]{Proposition}

\newtheorem{theorem}[thm]{Theorem}

\theoremstyle{remark}
\newtheorem{rem}[thm]{Remark}

\newcommand{\Trace}{\mathrm{Trace}}

\newcommand{\Sym}{\mathrm{Sym}}
\newcommand{\Hom}{\mathrm{Hom}}

\newcommand{\Swan}{\mathrm{Swan}}
\newcommand{\Fp}{\mathbb F_p}
\newcommand{\Fq}{\mathbb F_q}
\newcommand{\Fqm}{\mathbb F_{q^m}}
\newcommand{\CC}{\mathbb C}
\newcommand{\FFq}{{\overline{\mathbb F}_q}}

\newcommand{\QQ}{\overline{\mathbb Q}_\ell}
\newcommand{\R}{\mathrm R}

\newcommand{\Z}{\mathbb Z}
\newcommand{\Gal}{\mathrm {Gal}}
\newcommand{\AAA}{\mathbb A}

\newcommand{\PP}{\mathbb P}

\newcommand{\HH}{\mathrm H}

\newcommand{\LL}{{\mathcal L}}

\newcommand{\Aid}{\Aif}
\newcommand{\Aif}{\mathrm{Ai}_f}

\def\bb #1{ {\mathbb #1} }
\def\c #1{ {\mathcal #1} }

\title{$L$-functions of symmetric powers of the generalized Airy family of exponential sums}
\author{C. Douglas Haessig\footnote{Partially supported by NSF grant DMS-0901542} \and Antonio Rojas-Le\'on\thanks{Partially supported by MTM2007-66929 and FEDER} }

\date{\today}

\begin{document}
\maketitle

\begin{abstract}
This paper looks at the $L$-function of the $k$-th symmetric power of the $\overline{\bb Q}_\ell$-sheaf $\Aif $ over the affine line $\bb A^1_{\bb F_q}$ associated to the generalized Airy family of exponential sums. Using $\ell$-adic techniques, we compute the degree of this rational function and local factors at infinity.
\end{abstract}

\section{Introduction}

In this paper we study the $L$-function attached to the $k$-th symmetric power of the $\overline{\bb Q}_\ell$-sheaf $\Aif$ associated to the generalized Airy family of exponential sums. Symmetric powers appear in the proofs of many arithmetic problems. For instance, Deligne's proof \cite{Deligne} of the Ramanujan-Petersson conjecture relies on the construction of a Galois module coming from the $k$-th symmetric power of a certain $\ell$-adic sheaf. The Sato-Tate conjecture \cite{CHT} \cite{HST} \cite{T} relies on the analytic continuation of the $L$-function attached to the $k$-th symmetric power of an $\ell$-adic representation coming from an elliptic curve. Another equidistribution result concerning Kloosterman angles was proven by Adolphson \cite{Adolphson-distributionofangles-1989} using results of Robba's \cite{Robba-SymmetricPowersof-1986} on the $L$-function of the $k$-th symmetric power of the $\ell$-adic Kloosterman sheaf $\text{Kl}_2$. Symmetric powers also arise in the proof of Dwork's conjecture \cite{Wan-Dwork'sconjectureunit-1999} \cite{Wan-Higherrankcase-2000} \cite{Wan-Rankonecase-2000}. To begin, let us recall the general setup of an $L$-function of an $\ell$-adic representation.

Let $\bb F_q$ be the finite field of $q$ elements and characteristic $p$. Let $Y$ be a smooth, geometrically connected, open variety defined over $\bb F_q$; for instance, take $Y$ to be affine $s$-space $\bb A_{\bb F_q}^s$ or the torus $\bb G_m^s$. Denote its function field by $K$, and its corresponding absolute Galois group by $G_K := \mathrm{Gal}(K^{sep}/K)$. Let $V$ be a finite dimensional vector space over a finite extension field of $\bb Q_\ell$, where $\ell \not= p$. Let $\rho : G_K \rightarrow GL(V)$ be a continuous $\ell$-adic representation unramified on $Y$, and let $\c F$ be the corresponding lisse sheaf on $Y$. Define the $L$-function of $\rho$ on $Y$ by
\begin{equation}\label{E: GeneralLfunction}
L(Y, \rho, T) := \prod_{x \in |Y|} \frac{1}{\det(1 - \rho(\text{Frob}_x) T^{\text{deg}(x)})}.
\end{equation}
By the Lefschetz trace formula, this is a rational function whose zeros and poles may be described using \'etale cohomology with compact support:
\[
L(Y, \rho, T) = \prod_{i=0}^{2 \text{dim}(Y)} \det(1 - \text{Frob}_q T | H_c^i(Y \otimes \overline{\bb F}_q, \c F))^{(-1)^{i+1}}
\]
Given such a representation, we may construct new $L$-functions via operations such as tensor, symmetric, or exterior products. Natural questions about these new $L$-functions concern the determination of their degrees (Euler characteristic) and describing various properties about their zeros and poles. In this paper, we will focus on the symmetric powers of a particular family of exponential sums called the \emph{generalized Airy family}. Other families whose symmetric powers have been investigated are the Legendre family of elliptic curves \cite{Adolphson-$p$-adictheoryof-1976} \cite{Dwork-Heckepolynomials-1971} and the hyperKloosterman family \cite{FuWan-$L$-functionssymmetricproducts-2005} \cite{FuWan-TrivialfactorsL-functions-} \cite{Robba-SymmetricPowersof-1986}. We note that the former seems to have been motivated by Dwork's $p$-adic interest in the Ramanujan-Petersson conjecture.

The generalized Airy family is defined as follows. Let $f$ be a polynomial over $\bb F_q$ of degree $d$ with $p \nmid d$. Let $\psi$ be a nontrivial additive character on $\bb F_q$. For each $\bar t \in \overline{\bb F}_q$ define its degree by $\text{deg}(\bar t) := [\bb F_q(\bar t) : \bb F_q]$. It is well-known that the associated $L$-function of the sequence of exponential sums
\[
S_m(\bar t) := \sum_{x \in \bb F_{q^{m \> \text{deg}(\bar t)}}} \psi \circ \text{Tr}_{\bb F_{q^{m \> \text{deg}(\bar t)}} / \bb F_q}( f(x) + \bar t x) \quad \text{for } m = 1, 2, 3, \ldots
\]
is a polynomial of degree $d-1$:
\[
L( f, \bb A^1, \bar t; T) := \exp\left( \sum_{m=1}^\infty S_m(\bar t) \frac{T^m}{m} \right) = (1 - \pi_1(\bar t) T) \cdots (1-\pi_{d-1}(\bar t) T).
\]
As we will describe later, the relative cohomology of this family may be represented $\ell$-adically as a lisse sheaf of rank $d-1$ over $\bb A^1$ via Fourier transform. Let us denote this sheaf by $\Aif$. The $L$-function of the $k$-th symmetric power of $\Aif$ takes the form:
\[
M_k(f, T) := L( \bb A^1, \text{Sym}^k( \Aif ), T) := \prod_{t \in | \bb A^1 |} \prod_{a_1 + \cdots + a_{d-1} = k} (1 - \pi_1(t)^{a_1} \cdots \pi_{d-1}(t)^{a_{d-1}} T^{deg(t)})^{-1},
\]
where $|\bb A^1|$ denotes the set of closed points on $\bb A^1$. By the Lefschetz trace formula, $M_k(f, T)$ is a rational function. The $\ell$-adic sheaf $\Aif$ was extensively studied by N. Katz in \cite{Katz-monodromygroupsattached-1987}, where its monodromy group is determined and, as a consequence, an equidistribution result is obtained for the exponential sums in the family (\cite[Corollary 20]{Katz-monodromygroupsattached-1987}). From these results it follows that, for $p>2d-1$, $M_k(f, T)$ is in fact a polynomial. For $d=3$, a study of the monodromy group may be avoided using Adolphson's method \cite{Adolphson-distributionofangles-1989}.

Our first main result is the computation of the degree of $M_k(f,T)$ for $p > d$. The degree of the rational function $M_k(f, T)$ equals the $k$-th coefficient of a generating series which is explicitly given in Corollary \ref{degreeM}. Simplified formulas are given in section \ref{someexamples} for some particularly nice values of $f$ and $p$.

As an example of this theorem, consider the family generated by $f(x) = x^d$. Then the degree of $M_k(x^d, T)$ may be described as follows. Let $\zeta$ be a primitive $(d-1)$-th root of unity in $\overline{\bb F}_q$. Denote by $N_{d-1, k}$ the number of $(d-1)$-tuples $(a_0,a_1,\ldots,a_{d-2})$ of nonnegative integers such that $a_0+a_1+\cdots+a_{d-2}=k$ and $a_0 + a_1 \zeta + \cdots + a_{d-2} \zeta^{d-2} = 0$ in $\overline{\bb F}_q$.

\begin{theorem} With the notation defined above, we have
\[
\mathrm{deg} \> M_k(x^d,T) = \frac{1}{d-1} \left[ {{k+d-2} \choose {d-2}} - dN_{d-1,k} \right].
\]
\end{theorem}

It was conjectured in \cite{Haessig-$L$-functionsofsymmetric-} that $M_k(x^3, T)$ is a polynomial for all $p > 3$ since it was shown, in that paper, that $M_k(x^3, T)$ is a polynomial for every odd integer $k$, and also for every $k$ even with $k < 2p$. Surprisingly, for $p = 5$, $M_k(x^3, T)$ is \emph{not} a polynomial for infinitely many $k$. This was communicated to the first author by N. Katz and is a consequence of the geometric monodromy group of $\text{Ai}_{x^3}$ being finite.

\begin{theorem}\label{T: FunEqu}
Suppose $p > 2d-1$. Then $M_k(f, T)$ is a polynomial which may be factored into a product $Q_k(f,T) P_k(f,T)$, where $P_k(f,T)$ satisfies the functional equation
$$
{P_k(f,T)} = c T^{\mathrm{deg}(P_k)} \overline{P_k(f, 1/q^{k+1}T)} \qquad \text{with } |c| = q^{\mathrm{deg}(P_k)(k+1)/2}
$$
and $Q_k(f, T)$ has reciprocal roots of weight $\leq k$. Furthermore, writing $f(x) = \sum_{i=0}^d c_i x^i$, if we assume $\bb F_q$ contains the $2(d-1)$-th roots of $-d c_d$ then an explicit description of $Q_k(f, T)$ may be given; see Corollary \ref{localLfunction}.
\end{theorem}

Describing the $p$-adic behavior of the reciprocal roots of $M_k(f, T)$ is also of interest. Motivation for such a study comes from Wan's reciprocity theorem \cite{Wan-Dimensionvariationof-1998} of the Gouv\^ea-Mazur conjecture \cite{GouveaMazur} on the slopes of modular forms; see \cite{Adolphson-$p$-adictheoryof-1976} for the connection between symmetric powers of the Legendre crystal with Hecke polynomials. At present we are able to prove the following improvement to \cite{Haessig-$L$-functionsofsymmetric-}. Assume $p \geq 7$, $k$ is odd and $k < p$. Write $M_k(x^3, T) = 1 + c_1 T + \cdots + c_{(k+1)/2} T^{(k+1)/2}$. Then the $q$-adic Newton polygon lies on or above the quadratic function $\frac{1}{3}( m^2 + m + k m )$ for $m = 0, 1, 2, \ldots, \frac{k+1}{2}$. Furthermore, as a consequence of the functional equation, the endpoints of the $q$-adic Newton polygon of $M_k(x^3, T)$ coincide with this lower bound. We will prove this in a separate paper.

\bigskip\noindent{\bf Acknowledgments.} We would like to thank Nicholas Katz and Steven Sperber for their very helpful comments.

\section{Cohomological interpretation of $M_k(f, T)$}

In this section we will study the generalized Airy family of exponential sums from the point of view of $\ell$-adic cohomology. We will do so by studying the sheaf $\mathrm{Ai}_f$ that represents this family on the affine line $\AAA^1$ over the given finite field $\Fq$. We begin by observing that the map $\Fq\to \CC$ given by $t\mapsto \sum_{x\in \Fq}\psi(f(x)+tx)$ is the Fourier transform with respect to $\psi$, in the classical sense, of the map $t\mapsto \psi(f(t))$. This will translate, in the cohomological sense, to the fact that $\mathrm{Ai}_f$ is the Fourier transform, in the sheaf-theoretical sense, of the $\QQ$-sheaf that represents the latter map, which is just the pull-back of the Artin-Schreier sheaf associated to $\psi$ via the map given by $f$. Let us be more precise.

The polynomial $f$ naturally defines a morphism, also denoted by $f:\AAA^1_{\Fq}\to\AAA^1_{\Fq}$. Let ${\mathcal L}_\psi$ be the Artin-Schreier sheaf on $\AAA^1_{\Fq}$ associated to $\psi$ (cf. \cite[1.7]{deligne-traces}). For every finite extension $\Fqm$ of $\Fq$, every $t \in \AAA^1(\Fqm)=\Fqm$ and every geometric point $\bar t$ over $t$, we have $\Trace(\mathrm{Frob}_t|{\mathcal L}_{\psi,\bar t}) = \psi(\Trace_{\Fqm / \Fq}(t))$, where $\mathrm{Frob}_t$ denotes a geometric Frobenius element at $t$. Consider the pullback ${\mathcal L}_{\psi(f)}:=f^\star{\mathcal L}_\psi$.

By \cite[Theorem 17]{Katz-monodromygroupsattached-1987}, for $d\geq 2$ the Fourier transform with respect to $\psi$ of ${\mathcal L}_{\psi(f)}$ (which, in principle, is an element of the derived category ${\mathcal D}^b_c(\AAA^1,\QQ)$) is in fact a (shifted) lisse sheaf on $\AAA^1$, of rank $d-1$ and with $d/(d-1)$ as its single slope at infinity. Its Swan conductor is therefore $d$. Let us denote this sheaf by $\mathrm{Ai}_f=\R^1\pi_{t!}\LL_{\psi(f(x)+tx)}$, where $\pi_t:\AAA^2\to\AAA^1$ is the projection $(x,t)\mapsto t$. For every finite extension $\Fqm$ of $\Fq$, every $t\in \Fqm$ and every geometric point $\bar t$ over $t$ we have, denoting $\psi_m=\psi\circ\Trace_{\Fqm/\Fq}$:
$$
\Trace(\text{Frob}_t|({\mathrm{Ai}_f})_{\bar t}) = -\sum_{x\in \Fqm}\psi_m(f(x)+tx).
$$
The characteristic polynomial of the action of a geometric Frobenius element $\text{Frob}_t$ at $t$ on the stalk of $\mathrm{Ai}_f$ at a geometric point over $t$ has the form
$$
L(\Aif , t, T)=(1-\pi_1(t)T)\cdots(1-\pi_{d-1}(t)T)
$$
where $\pi_i(t)$ is a Weil algebraic number of weight $1$ (i.e. all its complex conjugates have absolute value $q^{1/2}$) and $\sum_{x\in \Fqm}\psi_m(f(x)+tx)=-\sum_i \pi_i(t)^m$ for all $m\geq 1$. Its $k$-th ``symmetric power'' is given by
$$
L(k; \Aif , t, T) := \prod_{a_1+\cdots+a_{d-1}=k} (1-\pi_1(t)^{a_1} \cdots \pi_{d-1}(t)^{a_{d-1}}T).
$$
These are the local factors of the $L$-function of the $k$-th symmetric power of $\Aif$, which is given by the infinite product
$$
M_k(f,T):=\prod_{t\in|\AAA^1|}L(k; \Aif , t, T^{\deg(t)})^{-1}
$$
The Lefschetz trace formula demonstrates that the zeros and poles of $M_k(f,T)$ may be described in terms of cohomology:
$$
M_k(f,T)=\prod_{i=0}^2 \det(1 - \text{Frob} \> T| \HH^i_c(\AAA^1_{\FFq},\Sym^k\Aif))^{(-1)^{i+1}}.
$$
Since $\Sym^k\Aif$ is a lisse sheaf on the affine line, we have $\HH^0_c(\AAA^1_{\FFq},\Sym^k\Aif)=0$, and the previous formula simplifies to
$$
M_k(f,T) = \frac{\det(1 - \text{Frob} \> T | \HH^1_c(\AAA^1_{\FFq},\Sym^k\Aif))}{\det(1 - \text{Frob} \> T| \HH^2_c(\AAA^1_{\FFq},\Sym^k\Aif))}.
$$

On the other hand, $\HH^2_c(\AAA^1_{\FFq},\Sym^k\Aif)$ is just the space of co-invariants of the sheaf $\Sym^k{\Aif}$, regarded as a representation of the fundamental group $\pi_1(\AAA^1_{\FFq})$, which is the $k$-th symmetric power of $\Aif$ regarded as a representation of the same group. This is the same as the space of co-invariants for its monodromy group, which is defined to be the Zariski closure of its image in the group of automorphisms of the generic stalk of $\Aif$, isomorphic to $\mathrm{GL}(d-1):=\mathrm{GL}(d-1,\QQ)$. By \cite[Theorem 19]{Katz-monodromygroupsattached-1987}, for $p > 2d-1$ the geometric monodromy group of $\Aif$ is either $\text{SL}(d-1)$ for $d$ even, or $\text{Sp}(d-1)$ for $d$ odd if $c_{d-1}=0$ and $\mu_p\cdot\text{SL}(d-1)$ for $d$ even or $\mu_p\cdot\text{Sp}(d-1)$ for $d$ odd if $c_{d-1}\neq 0$ (where $f(x)=\sum_{i=0}^d c_i x^i$). In either case, its $k$-th symmetric power is still an irreducible representation of rank ${d+k-2}\choose{d-2}$ of the monodromy group (because it is an irreducible representation of its subgroup $\text{SL}(d-1)$ or $\text{Sp}(d-1)$), and in particular the space of co-invariants vanishes. More generally, it was proven by O. \v{S}uch (\cite[Proposition 1.6]{such2000monodromy}) that, for $p>2$, either $\Aif$ has finite monodromy or its monodromy group contains $\text{SL}(d-1)$ or $\text{Sp}(d-1)$. In order to rule out the finite monodromy case for $p\leq 2d-1$ one may use for instance \cite[Proposition 8.14.3]{katz-esde}, which implies that $\Aif$ has finite monodromy if and only if for every element $t\in\overline{{\mathbb F}_{q}}$ the Newton polygon of the $L$-function associated to the exponential sum $\sum\psi(f(x)+tx)$ has a single slope.

Consequently, we have the following:

\begin{thm} If $\Aif$ does not have finite monodromy (e.g. if $p>2d-1$), the $L$-function of the $k$-th symmetric power of $\Aif$ is a polynomial:
$$
M_k(f,T) = \det(1 - \mathrm{Frob} \> T | \HH^1_c(\AAA^1_{\FFq},\Sym^k{\Aif}))
$$
\end{thm}

While it is tempting to believe that $M_k(f,T)$ is always a polynomial, this is not true, as mentioned in the introduction. In fact, the monodromy group can be finite in certain cases; for instance when $p=5$ and $f(x)=x^3$, as proven in \cite{Katz-$G_2$andHypergeometric-2007}. In such cases, $\HH^2_c(\AAA^1_{\FFq},\Sym^k{\Aif})$ will be non-trivial for infinitely many values of $k$, and consequently $M_k(f,T)$ will have a denominator.

\begin{rem}
Arithmetic difficulties often arise when the characteristic $p$ is small compared to $d$, as demonstrated above by the link between the finiteness of the monodromy group when $p \leq 2d-1$ and the Newton polygons of the fibres of the family.  By the functional equation, if we denote by $\text{NP}_1(t)$ the slope of the first line segment of the Newton polygon of the fibre $t$ then $\text{NP}_1(t) \leq 1/2$ with equality  if and only if the Newton polygon is a single line segment. 
If $p \equiv 1$ modulo $d$, and in particular when $p = d+1$, then by \cite[Theorem 3.11]{sperber1986p} the Newton polygon of every fibre equals the $q$-adic Newton polygon of the polynomial $\prod_{i=1}^{d-1} ( 1 - q^{i/d} T)$. Thus, $\text{NP}_1(t) = 1/d$ and so the monodromy group is infinite when $p = d+1 > 3$.  

Let $[f(x)]_{x^N}$ denote the coefficient of $x^N$ in $f(x)$. Suppose $\frac{d}{2} + 1 < p \leq 2d-1$ and $f$ has coefficients over $\bb F_p$. By \cite[Theorem 2]{ScholtenZhu}, if $[(f(x)+tx)^{\lceil \frac{p-1}{d} \rceil}]_{x^{p-1}} \not\equiv 0$ modulo $p$ for some $0 \leq t \leq p-1$, then $\text{NP}_1(t) \leq \left\lceil \frac{p-1}{d} \right\rceil / (p-1)$ for those $t$. Hence, the mondromy group is infinite when such a $t$ exists and $d \geq 3$ and $p \geq 7$. Their argument can be extended to show the following. Let $( h(x) )_s := h(x) (h(x) - 1) \cdots (h(x) - s + 1)$. Define the linear operator $U: \bb F_p[x] \rightarrow \bb F_p$ by sending $x^n$ to 0 if $(p-1) \nmid n$ and 1 otherwise. Let $c_s := U( (f(x)+tx)_s)$. Suppose $c_1 \equiv \cdots \equiv c_{k-1} \equiv 0$ modulo $p$ and $c_k \not\equiv 0$ mod $p$ for some $t$, then $\text{NP}_1(t) \leq \frac{k}{p-1}$. Hence, if this happens for some $k <  (p-1)/2$ then the mondromy group is infinite. For example, for $d > p-1$ and $f(x) = x^d + x^{p-1}$ then $c_1 = 1$ and hence the monodromy group is infinite for $p \geq 5$.

Lastly,  we mention the case when $d=4$, $p=7$ and $f \in \bb F_q[x]$ is not of the form $(x+a)^4+bx+c$, then by \cite[Theorem 4.6]{hong2001newton} the monodromy of $\Aif$ is not finite.
\end{rem}

\section{Computation of the degree of the $L$-function}

We will now study the degree of $M_k(f,T)$ when $p > d$. From the formula above we have \[
\mathrm{deg}(M_k(f, T)) = \dim(\HH^1_c(\AAA^1_{\FFq},\Sym^k{\Aif}))-\dim(\HH^2_c(\AAA^1_{\FFq},\Sym^k{\Aif})) = -\chi_c(\AAA^1_{\FFq},\Sym^k{\Aif}),
\]
where $\chi_c$ denotes the Euler characteristic with compact supports. Using the Grothendieck-N\'eron-Ogg-Shafarevic formula, we have then
\begin{align}\label{degree}
\mathrm{deg}(M_k(f, T)) &= \mathrm{Swan}_\infty(\Sym^k{\Aif}) - \mathrm{rank}(\Sym^k{\Aif}) \notag \\
&= \mathrm{Swan}_\infty(\Sym^k{\Aif}) - {{k+d-2}\choose{d-2}}.
\end{align}
In order to compute the Swan conductor of $\Sym^k{\Aif}$ we have to study the sheaf $\Aif$ as a representation of the inertia group $I_\infty$ of $\AAA^1_{\FFq}$ at infinity. Since $\LL_{\psi(f)}$ is lisse on $\AAA^1$, as a representation of the decomposition group at infinity we have $\Aif\cong {\mathcal F}_{\infty,\infty}(\LL_{\psi(f)})$, where ${\mathcal F}_{\infty,\infty}$ is the local Fourier transform as defined in \cite{laumon87}.

Recently, Fu \cite{fu2008} and, independently, Abbes and Saito \cite{abbes-saito} have given an explicit description of the different local Fourier transforms for a wide class of $\ell$-adic sheaves. We will mainly be using the description given in \cite{abbes-saito}, which works over an arbitrary (not necessarily algebraically closed) perfect base field, and therefore gives an explicit formula for $\Aif$ as a representation of the decomposition group $D_\infty$.

If $S_{(\infty)}$ is the henselization of the local ring of $\PP^1_{\Fq}$ at infinity with uniformizer $1/t$, the triple $(\LL_{\psi(f(t))},t,-f'(t))$ is a Legendre triple in the sense of \cite[Definition 2.16]{abbes-saito}. Therefore by \cite[Theorem 3.9]{abbes-saito} we conclude that, as a representation of $D_\infty$, $\Aif$ is isomorphic to
$$
(-f')_\star(\LL_{\psi(f(t))}\otimes\LL_{\psi(-tf'(t))}\otimes\LL_{\rho(\frac{1}{2}f''(t))}\otimes{\mathcal Q})=(-f')_\star(\LL_{\psi(f(t)-tf'(t))}\otimes\LL_{\rho(\frac{1}{2}f''(t))}\otimes{\mathcal Q})
$$
where $\rho$ is the unique character $I_\infty\to\QQ^\star$ of order $2$, $\LL_{\rho}$ the corresponding Kummer sheaf and ${\mathcal Q}$ is the pull-back of the character $\mathrm{Gal}(\FFq/\Fq)\to\QQ$ mapping the geometric Frobenius to the quadratic Gauss sum $g(\psi,\rho):=-\sum_{t\in \Fq^\star}\psi(t)\rho(t)$.

Write $f(t)=\sum_{i=0}^d c_it^i$. For simplicity, from now on we will assume that $\Fq$ contains the $2(d-1)$-th roots of $-dc_d$ (which can always be achieved by a finite extension of the base field). Following \cite[Proposition 3.1]{fu2008} we can find an invertible power series $\sum_{i\geq 0} r_i t^{-i}\in\Fq[[t^{-1}]]$ with $r_0^{d-1}=-dc_d$ such that $u(t):=t \sum_{i\geq 0} r_i t^{-i}$ is a solution to $f'(t)+u(t)^{d-1}=0$ (the other solutions being $\zeta u(t)$ for every $(d-1)$-th root of unity $\zeta$). The map $\phi:1/t\mapsto 1/u(t)$ defines an automorphism $S_{(\infty)}\to S_{(\infty)}$, and by construction $-f'=[d-1]\circ \phi$, where $[d-1]$ is the $(d-1)$-th power map. So $\Aif$ is isomorphic to
\begin{align*}
[d-1]_\star\phi_\star(\LL_{\psi(f(t)-tf'(t))}\otimes\LL_{\rho(\frac{1}{2}f''(t))}\otimes{\mathcal Q}) &= [d-1]_\star(\phi^{-1})^\star(\LL_{\psi(f(t)-tf'(t))}\otimes\LL_{\rho(\frac{1}{2}f''(t))}\otimes{\mathcal Q}) \\
&=[d-1]_\star(\LL_{\psi(f(v(t))+v(t)t^{d-1})}\otimes\LL_{\rho(\frac{1}{2}f''(v(t)))}\otimes{\mathcal Q}) \\
&= [d-1]_\star(\LL_{\psi(f(v(t))+v(t)t^{d-1})}\otimes\LL_{\rho(\frac{1}{2}f''(v(t)))})\otimes{\mathcal Q}
\end{align*}
since $[d-1]^\star{\mathcal Q}={\mathcal Q}$, where $v(t):=\phi^{-1}(t)=t\sum_{i\geq 0} s_i t^{-i}$.

Let $g(t)$ be the polynomial of degree $d$ obtained from $f(v(t))+v(t)t^{d-1}$ by removing the terms with negative powers of $t$. It is important to notice that the coefficients of $g$ are polynomials in the coefficients of $f$. More precisely, if we write $g(t)=\sum b_i t^i$, the coefficient $b_i$ is a polynomial in the coeficients $a_i,a_{i+1},\ldots,a_d$ of $f$. Since $\LL_{\psi(h(t))}$ is trivial as a representation of $D_\infty$ for any $h(t)\in t^{-1}\Fq[[t^{-1}]]$, we have an isomorphism $\LL_{\psi(f(v(t))+v(t)t^{d-1})}\cong\LL_{\psi(g(t))}$ as representations of $D_\infty$.

On the other hand, from $f'(v(t))+t^{d-1}=0$ we get $f''(v(t))v'(t)+(d-1)t^{d-2}=0$, so $\LL_{\rho(\frac{1}{2}f''(v(t)))}=\LL_{\rho(-\frac{d-1}{2}v'(t)t^{d-2})}$ . Since $v'(t)=\sum_{i\geq 0}(1-i)s_it^{-i}=s_0(1+\sum_{i\geq 2}(1-i)\frac{s_i}{s_0}t^{-i})$ and $1+\sum_{i\geq 2}(1-i)\frac{s_i}{s_0}t^{-i}$ is a square in $\Fq[[t^{-1}]]$, we have $\LL_{\rho(-\frac{d-1}{2}v'(t)t^{d-2})}=\LL_{\rho(-\frac{d-1}{2}s_0t^{d-2})}=\LL_{\rho(\frac{d(d-1)}{2}c_d(s_0t)^{d-2})}$ (since $s_0^{d-1}=-1/dc_d$). So we finally get
\begin{equation}\label{MonodromyAtInf}
\Aif\cong[d-1]_\star(\LL_{\psi(g(t))}\otimes\LL_{\rho^d(s_0t)})\otimes\LL_{\rho(d(d-1)c_d/2)}\otimes{\mathcal Q}.
\end{equation}

We can now easily compute the Swan conductor at infinity of its symmetric powers. By \cite[1.13.1]{Katz-GaussSumsKloosterman-1988}, $$
\Swan_\infty\Sym^k\Aif=\frac{1}{d-1}\Swan_\infty[d-1]^\star\Sym^k\Aif=\frac{1}{d-1}\Swan_\infty\Sym^k[d-1]^\star\Aif
$$

\begin{lem}\label{induced} Let $\zeta$ be a primitive $(d-1)$-th root of unity if $\Fq$, $I_\infty^{d-1}$ the unique closed subgroup of $I_\infty$ of index $d-1$. As a representation of $I_\infty^{d-1}$, the restriction $[d-1]^\star\Aif$ of $\Aif$ is isomorphic to the direct sum $$\bigoplus_{i=0}^{d-2}\LL_{\psi(g(\zeta^it))}\otimes\LL_{\rho^d(s_0\zeta^i t)}\cong \bigoplus_{i=0}^{d-2}\LL_{\psi(g(\zeta^it))}\otimes\LL_{\rho^d(t)}$$.\end{lem}

\begin{proof}
Since $(\zeta^i)^\star\LL_{\psi(g)}=\LL_{\psi(g(\zeta^it))}$, $(\zeta^i)^\star\LL_{\rho^d(s_0t)}=\LL_{\rho^d(s_0\zeta^it)}$ and $[d-1]\circ \zeta^i =[d-1]$ for every $i$, we have $[d-1]_\star(\LL_{\psi(g(\zeta^it))}\otimes\LL_{\rho^d(s_0\zeta^it)})=[d-1]_\star(\LL_{\psi(g(t))}\otimes\LL_{\rho^d(s_0t)})$, and therefore by Frobenius reciprocity $\Hom_{I_\infty^{d-1}}([d-1]^\star\Aif,\LL_{\psi(g(\zeta^it))}\otimes\LL_{\rho^d(s_0\zeta^it)})=\Hom_{I_\infty}(\Aif,[d-1]_\star(\LL_{\psi(g(\zeta^it))}\otimes\LL_{\rho^d(s_0\zeta^it)}))=\Hom_{I_\infty}(\Aif,\Aif)\cong \QQ$ since the latter is an irreducible representation of $I_\infty$. So for every $i$, $\LL_{\psi(g(\zeta^it))}\otimes\LL_{\rho^d(s_0\zeta^it)}$ is a subrepresentation of $[d-1]^\star\Aif$.

Now $\LL_{\psi(g(\zeta^it))}\otimes\LL_{\rho^d(s_0\zeta^it)}$ and $\LL_{\psi(g(\zeta^jt))}\otimes\LL_{\rho^d(s_0\zeta^jt)}$ are isomorphic if and only if $\LL_{\psi(g(\zeta^it))}$ and $\LL_{\psi(g(\zeta^jt))}$ are, if and only if $g(\zeta^it)-g(\zeta^jt)=h^p-h$ for some $h\in\FFq[t]$. Since $p>d$, this can only happen if $g(\zeta^it)=g(\zeta^jt)$. Comparing the highest degree coefficients we conclude that $\zeta^i$ and $\zeta^j$ must be equal. Therefore the direct sum of the $\LL_{\psi(g(\zeta^it))}\otimes\LL_{\rho^d(s_0\zeta^it)}$ for $i=0,\ldots,d-2$ injects into $[d-1]^\star\Aif$ and we conclude that it must be isomorphic to it, since they have the same rank.
\end{proof}

Consequently, we have an isomorphism of $\QQ[I_\infty]$-modules
$$
\Sym^k[d-1]^\star\Aif\cong\bigoplus_{a_0+a_1+\cdots+a_{d-2}=k}\LL_{\psi(\sum_{i=0}^{d-2} a_i g(\zeta^it))}\otimes\LL_{\rho^{dk}(t)}.
$$
For every finite subset $I\subset\Z$ and every integer $k\geq 0$ define
$$
S_{d-1}(k,I):=\{(a_0,\ldots,a_{d-2})\in\Z^{d-1}_{\geq 0} | a_0+a_1+\cdots+a_{d-2}=k, a_0+a_1\zeta^i+\cdots+a_{d-2}\zeta^{i(d-2)}=0 \text{ for every }i\in I\}
$$

It is clear from the definition that $S_{d-1}(k,I)=S_{d-1}(k,I')$ if $\phi(I)=\phi(I')$, where $\phi:\Z\to\Z/(d-1)\Z$ is reduction modulo $d-1$. Also, $S_{d-1}(k,I)=\emptyset$ if $p$ does not divide $k$ and $I\cap (d-1)\Z\neq\emptyset$. The number of elements in $S_{d-1}(k,I)$ can be conveniently expressed in terms of a generating function:

\begin{lem}
 Let $F_{d-1}(I;T):=\sum_{k=0}^{\infty} \# S_{d-1}(k,I)T^k$. Then
$$
F_{d-1}(I;T)=\frac{1}{q^{\#I}}\sum_{\gamma\in{(\Fq)^I}}\prod_{j=0}^{d-2}(1-\psi(\sum_{i\in I}\gamma_i\zeta^{ji})T)^{-1}
$$
where $\psi$ is any non-trivial additive character of $\Fq$.
\end{lem}

\begin{proof}
 From the definition,
$$
F_{d-1}(I;T)=\sum_{(a_0,\ldots,a_{d-2})\in\Z^{d-1}_{\geq 0}}\prod_{i\in I}\delta(a_0+a_1\zeta^i+\cdots+a_{d-2}\zeta^{i(d-2)})T^{a_0+a_1+\cdots+a_{d-2}}
$$
where $\delta(a)=1$ if $a=0$, $0$ otherwise. Equivalently, $\delta(a)=\frac{1}{q}\sum_{\gamma\in \Fq}\psi(\gamma a)$. So we get
\begin{align*}
F_{d-1}(I;T) &= \sum_{(a_0,\ldots,a_{d-2})\in\Z^{d-1}_{\geq 0}}\prod_{i\in I}\frac{1}{q}\sum_{\gamma_i\in \Fq}\psi(\gamma_i(a_0+a_1\zeta^i+\cdots+a_{d-2}\zeta^{i(d-2)}))T^{a_0+a_1+\cdots+a_{d-2}} \\
&= \sum_{(a_0,\ldots,a_{d-2})\in\Z^{d-1}_{\geq 0}}\sum_{\gamma\in (\Fq)^I}\frac{1}{q^{\# I}}\left(\prod_{i\in I}\psi(\gamma_i a_0)\right)T^{a_0}\left(\prod_{i\in I}\psi(\gamma_i a_1 \zeta^i)\right)T^{a_1}\cdots\left(\prod_{i\in I}\psi(\gamma_i a_{d-2} \zeta^{(d-2)i})\right)T^{a_{d-2}} \\
&= \frac{1}{q^{\# I}}\sum_{\gamma\in (\Fq)^I}\sum_{(a_0,\ldots,a_{d-2})\in\Z^{d-1}_{\geq 0}}\psi(\sum_{i\in I}\gamma_i )^{a_0}T^{a_0}\psi(\sum_{i\in I}\gamma_i \zeta^i)^{a_1}T^{a_1}\cdots\psi(\sum_{i\in I}\gamma_i  \zeta^{(d-2)i})^{a_{d-2}}T^{a_{d-2}} \\
&= \frac{1}{q^{\# I}}\sum_{\gamma\in (\Fq)^I}\left(\sum_{a_0\in\Z_{\geq 0}}\psi(\sum_{i\in I}\gamma_i )^{a_0}T^{a_0}\right)\left(\sum_{a_1\in\Z_{\geq 0}}\psi(\sum_{i\in I}\gamma_i \zeta^i)^{a_1}T^{a_1}\right)\cdots\left(\sum_{a_{d-2}\in\Z_{\geq 0}}\psi(\sum_{i\in I}\gamma_i  \zeta^{(d-2)i})^{a_{d-2}}T^{a_{d-2}}\right) \\
&= \frac{1}{q^{\#I}}\sum_{\gamma\in{(\Fq)^I}}\prod_{j=0}^{d-2}(1-\psi(\sum_{i\in I}\gamma_i\zeta^{ji})T)^{-1}.
\end{align*}
\end{proof}

Write $g(t)=\sum_{j=0}^d b_j t^j$, and let $J=\{1\leq j\leq d|b_j\neq 0\}$ and $J_{\geq j}:=J\cap \{j,j+1,\ldots,d\}$ for every $j\in\{1,\ldots,d,d+1\}$. We have
\begin{align*}
\Swan_\infty\Sym^k[d-1]^\star\Aif &=\sum_{a_0+a_1+\cdots+a_{d-2}=k}\Swan_\infty\LL_{\psi(\sum_{i=0}^{d-2} a_i g(\zeta^it))}\otimes\LL_{\rho^{dk}(t)} \\
&=\sum_{a_0+a_1+\cdots+a_{d-2}=k}\deg(\sum_{i=0}^{d-2} a_i g(\zeta^i t))
\end{align*}
and
$$
\sum_{i=0}^{d-2} a_i g(\zeta^i t)=\sum_{i=0}^{d-2} a_i\sum_{j=0}^d b_j\zeta^{ij} t^j=\sum_{j=0}^d(b_j\sum_{i=0}^{d-2}\zeta^{ij})t^j
$$
so its degree is the greatest $j$ such that $b_j\sum_{i=0}^{d-2}\zeta^{ij}\neq 0$. Therefore we get
\begin{align*}
(d-1)\Swan_\infty\Sym^k\Aif &= \Swan_\infty\Sym^k[d-1]^\star\Aif \\
&= \sum_{j\in J}j\cdot(\# S_{d-1}(k,J_{\geq j+1})-\# S_{d-1}(k,J_{\geq j})) \\
&= d{{k+d-2}\choose{d-2}}-\sum_{j\in J} h(j)\cdot \# S_{d-1}(k,J_{\geq j})
\end{align*}
where $h(j):=j-\sup(J - J_{\geq j})$ is the ``gap'' between the $t^j$ term and the next lower degree term in $g(t)$. Taking the corresponding generating function we get the formula

\begin{cor}
 Let $G(f;T):=\sum_{k=0}^\infty (\Swan_\infty\Sym^k\Aif) T^k$, then
$$
G(f;T)=\frac{d}{(d-1)(1-T)^{d-1}}-\frac{1}{d-1}\sum_{j\in J} h(j)\cdot F_{d-1}(J_{\geq j};T)
$$
\end{cor}

Using the previous formula for the degree, we deduce

\begin{cor}\label{degreeM}
The degree of $M_k(f;T)$ is the $k$-th coefficient of the power series expansion of
$$
\frac{1}{(d-1)(1-T)^{d-1}}-\frac{1}{d-1}\sum_{j\in J} h(j)\cdot F_{d-1}(J_{\geq j};T).
$$

\end{cor}

\begin{cor}
 For every $J\subset \{1,\ldots,d-1\}$, let ${\mathcal P}_d(J)$ be the subspace of the affine space ${\mathcal P}_d$ of polynomials of degree $d$ over $k$ such that $b_j=0$ if and only if $j\in J$. The sets $\{{\mathcal P}_d(J)|J\subseteq \{1,\ldots,d-1\}\}$ define a stratification of ${\mathcal P}_d$ such that the degree of $M_k(f;T)$ is constant in each stratum.
\end{cor}

\section{The trivial factor}

Suppose $p>d$ and the monodromy of $\Aif$ is not finite. We will now study the weights of the (reciprocal) roots of the polynomial $M_k(f,T)$. Let us first consider the easier case where $d$ is even, and therefore $\Aif$ is isomorphic to $[d-1]_\star\LL_{\psi(g(t))}\otimes\LL_{\rho(d(d-1)c_d/2)}\otimes{\mathcal Q}$ as a representation of $D_\infty$. Let $D^{d-1}_\infty=\Gal(\overline{\Fq((1/t))}/\Fq((1/t^{1/(d-1)})))$, denote by $\alpha:D_\infty^{d-1}\to\QQ^\star$ the character corresponding to the sheaf $\LL_{\psi(g)}$, and let $b\in I_\infty$ be a generator of the cyclic group $D_\infty/D_\infty^{d-1}\cong I_\infty/I_\infty^{d-1}$. By the explicit description of induced representations, there is a basis $\{v_0,\ldots,v_{d-2}\}$ of the underlying vector space $V$ such that $a\cdot v_0=\alpha(a)v_0$ for every $a\in I_\infty^{d-1}$ and $b\cdot v_i=v_{i+1}$ for $i=0,\ldots,d-3$. Then $b\cdot v_{d-2}=b^{d-1}\cdot v_0=\alpha(b^{d-1})v_0$. Replacing $b$ by $a^{-1}b$, where $a\in I_\infty^{d-1}$ is an element such that $\alpha(a)^{d-1}=\alpha(b^{d-1})$ (which is always possible since the values of $\alpha$ are the $p$-th roots of unity and $d-1$ is prime to $p$ since $p > d$) we may assume without loss of generality that $\alpha(b^{d-1})=1$.

Furthermore, for any $a\in I_\infty^{d-1}$ we have $a\cdot v_i=(ab^i)\cdot v_0=(b^ib^{-i}ab^i)\cdot v_0=b^i\cdot\alpha(b^{-i}ab^i)v_0=\alpha(b^{-i}ab^i)v_i$. So the restriction of $\Aif$ to $D_\infty^{d-1}$ is the direct sum of the characters $a\mapsto\alpha_i(a):=\alpha(b^{-i}ab^i)$. But we already know that it is the direct sum of the characters associated to the sheaves $\LL_{\psi(g(\zeta^i t))}\otimes\LL_{\rho(d(d-1)c_d/2)}\otimes{\mathcal Q}$, so these two sets of characters are identical. Replacing $b$ by a suitable power of itself we may assume that $\alpha_i$ is the character associated to $\LL_{\psi(g(\zeta^it))}\otimes\LL_{\rho(d(d-1)c_d/2)}\otimes{\mathcal Q}$. In particular, $\prod_{i=0}^{d-2}\alpha_i^{a_i}$ is geometricaly trivial (that is, trivial on $I_\infty^{d-1}$) if and only if $\sum a_i g (\zeta^it)$ is a constant in $\Fq[t]$, that is, if and only if $\sum a_i \zeta^{ij}=0$ for every $j\in J$.

We turn now to the case $d$ odd. Let $\chi$ be a multiplicative character of $\Fq$ of order $2(d-1)$ (which exists, since we are assuming that $\Fq$ contains the $2(d-1)$-th roots of unity). Then by the projection formula $\Aif$ is isomorphic to $[d-1]_\star(\LL_{\psi(g(t))}\otimes\LL_{\rho(s_0t)})\otimes\LL_{\rho(d(d-1)c_d/2)}\otimes{\mathcal Q}\cong([d-1]_\star\LL_{\psi(g(t))})\otimes\LL_{\chi(s_0t)}\otimes\LL_{\rho(d(d-1)c_d/2)}\otimes{\mathcal Q}$. Let $\alpha_i:D_\infty^{d-1}\to\QQ^\star$ (respectively $\beta:D_\infty\to\QQ^\star$) be the character corresponding to the sheaf $\LL_{\psi(g(\zeta^i t))}$ (resp. $\LL_{\chi(s_0t)}$). Proceeding as in the $d$ even case, we find a generator $b\in I_\infty$ of $D_\infty/D_\infty^{d-1}$ and a basis $\{v_0,\ldots,v_{d-2}\}$ of $V$ such that $a\cdot v_i=\alpha_i(a)\beta(a)v_i$ for $a\in D^{d-1}_\infty$ and $b\cdot v_i=\beta(b)v_{i+1}$ for $i=0,\ldots,d-3$, $b\cdot v_{d-2}=\beta(b)v_0$. In this case, $\prod_{i=0}^{d-2}\alpha_i^{a_i}\beta^{a_i}$ is trivial on $I^{d-1}_\infty$ if and only if $\sum a_i g (\zeta^it)$ is a constant in $\Fq[t]$ and $\sum a_i$ is even (since $\alpha_i$ has order $p$ and $\beta$ restricted to $I^{d-1}_\infty$ has order $2$).

We can now compute the dimension of the invariant subspace of the action of $I_\infty$ on $\Sym^k\Aif$, in very much the same way it is done for the Kloosterman sheaf in \cite[Lemma 2.1]{FuWan-$L$-functionssymmetricproducts-2005}. Its underlying vector space is $\Sym^k V$. An element $w$ is given by a linear combination
$$
w=\sum_{a_0+\cdots+a_{d-2}=k} c_{a_0\cdots a_{d-2}}v_0^{a_0}\cdots v_{d-2}^{a_{d-2}}.
$$

In the $d$ even case we have
$$
a\cdot\sum_{a_0+\cdots+a_{d-2}=k} c_{a_0\cdots a_{d-2}}v_0^{a_0}\cdots v_{d-2}^{a_{d-2}}
=\sum_{a_0+\cdots+a_{d-2}=k} c_{a_0\cdots a_{d-2}}(\alpha_0^{a_0}\cdots\alpha_{d-2}^{a_{d-2}})(a)v_0^{a_0}\cdots v_{d-2}^{a_{d-2}}
$$
for $a\in I_\infty^{d-1}$ and
$$
b\cdot \sum_{a_0+\cdots+a_{d-2}=k} c_{a_0\cdots a_{d-2}}v_0^{a_0}\cdots v_{d-2}^{a_{d-2}}=
\sum_{a_0+\cdots+a_{d-2}=k} c_{a_0\cdots a_{d-2}}v_1^{a_0}v_2^{a_1}\cdots v_0^{a_{d-2}}.
$$
So $w$ is fixed by $I_\infty$ if and only if the character $\alpha_0^{a_0}\cdots\alpha_{d-2}^{a_{d-2}}$ is trivial whenever $c_{a_0\cdots a_{d-2}}\neq 0$ and $c_{a_0\cdots a_{d-2}}=c_{a_{d-2}a_0\cdots a_{d-3}}$ for all $a_0,\ldots,a_{d-2}$. A basis for the invariant subspace is thus given by all distinct sums of the form (setting $v_{d-1+l}:=v_l$ for all $l\geq 0$):
$$
\sum_{j=0}^{d-2} v_j^{a_0}v_{j+1}^{a_1}\cdots v_{j+d-2}^{a_{d-2}}
$$
for all $a_0,\ldots,a_{d-2}$ such that $\alpha_0^{a_0}\cdots\alpha_{d-2}^{a_{d-2}}$ is trivial, that is, such that $\sum a_i\zeta^{ij}=0$ in $\Fq$ for every $j\in J$.

In the $d$ odd case we get
$$
g\cdot\sum_{a_0+\cdots+a_{d-2}=k} c_{a_0\cdots a_{d-2}}v_0^{a_0}\cdots v_{d-2}^{a_{d-2}}
=\sum_{a_0+\cdots+a_{d-2}=k} c_{a_0\cdots a_{d-2}}(\alpha_0^{a_0}\cdots\alpha_{d-2}^{a_{d-2}})(g)\beta^k(g)v_0^{a_0}\cdots v_{d-2}^{a_{d-2}}
$$
for $g\in I_\infty^{d-1}$ and
$$
h\cdot \sum_{a_0+\cdots+a_{d-2}=k} c_{a_0\cdots a_{d-2}}v_0^{a_0}\cdots v_{d-2}^{a_{d-2}}=
\sum_{a_0+\cdots+a_{d-2}=k} c_{a_0\cdots a_{d-2}}\beta(h)^k v_1^{a_0}v_2^{a_1}\cdots v_0^{a_{d-2}}.
$$
So $w$ is fixed by $I_\infty$ if and only if the character $\alpha_0^{a_0}\cdots\alpha_{d-2}^{a_{d-2}}\beta^k$ of $I_\infty^{d-1}$ is trivial whenever $c_{a_0\cdots a_{d-2}}\neq 0$ and $c_{a_0\cdots a_{d-2}}=c_{a_{d-2}a_0\cdots a_{d-3}}\beta(h)^k$ for all $a_0,\ldots,a_{d-2}$. Since all $\alpha_i$'s have order $p$ and the restriction of $\beta$ to $I_\infty^{d-1}$ has order $2$, $\alpha_0^{a_0}\cdots\alpha_{d-2}^{a_{d-2}}\beta^k$ is trivial if and only if both $\alpha_0^{a_0}\cdots\alpha_{d-2}^{a_{d-2}}$ and $\beta^k$ are trivial as characters of $I_\infty^{d-1}$, that is, if and only if $\sum a_i\zeta^{ij}=0$ in $\Fq$ for every $j\in J$ and $k$ is even. In particular, there are no non-zero invariants for $I_\infty$ if $k$ is odd. If $k$ is even, a generating set for the invariant subspace is given by all distinct sums of the form
$$
\sum_{j=0}^{d-2} \beta(h)^{jk}v_j^{a_0}v_{j+1}^{a_1}\cdots v_{j+d-2}^{a_{d-2}}
$$
for all $a_0,\ldots,a_{d-2}$ such that $\sum a_i\zeta^{ij}=0$ in $\Fq$ for every $j\in J$. Let $r$ be the size of the orbit of $(a_0,\ldots,a_{d-2})$ under the action of ${\mathbb Z}/(d-1){\mathbb Z}$ by cyclic permutations. If $r\neq d-1$, we can write
$$
\sum_{j=0}^{d-2} \beta(h)^{jk}v_j^{a_0}v_{j+1}^{a_1}\cdots v_{j+d-2}^{a_{d-2}}=\sum_{j=0}^{r-1} \beta(h)^{jk}(1+\beta(h)^{rk}+\cdots+\beta(h)^{\left(\frac{d-1}{r}-1\right)rk})v_j^{a_0}v_{j+1}^{a_1}\cdots v_{j+d-2}^{a_{d-2}}.
$$
Notice that $k$ must be a multiple of $\frac{d-1}{r}$, since $k=\sum_{i=0}^{d-2}a_i=\frac{d-1}{r}\sum_{i=0}^{r-1}a_i$. If $\frac{rk}{d-1}$ is odd we have
$$
1+\beta(h)^{rk}+\cdots+\beta(h)^{\left(\frac{d-1}{r}-1\right)rk}=\frac{1-\beta(h)^{(d-1)k}}{1-\beta(h)^{rk}}=0,
$$
so the above sum vanishes. On the other hand, if $\frac{rk}{d-1}$ is even it is clear that the element
$$
\sum_{j=0}^{d-2} \beta(h)^{jk}v_j^{a_0}v_{j+1}^{a_1}\cdots v_{j+d-2}^{a_{d-2}}=\frac{d-1}{r}\sum_{j=0}^{r-1} \beta(h)^{jk}v_j^{a_0}v_{j+1}^{a_1}\cdots v_{j+d-2}^{a_{d-2}}
$$
is non-zero, and to different orbits correspond different elements. To summarize, we have

\begin{prop} Let $T_{d-1}(k,J)$ be the set of orbits of the action of ${\mathbb Z}/(d-1){\mathbb Z}$ on the set $S_{d-1}(k,J)$ by cyclic permutations, and let $U_{d-1}(k,J)$ be the subset of orbits such that $\frac{rk}{d-1}$ is even, where $r$ is their cardinality. If $d$ is even, the invariant subspace of the representation $\Sym^k\Aid$ of $I_\infty$ has dimension $\#T_{d-1}(k,J)$. If $d$ is odd and $k$ is even, it has dimension $\#U_{d-1}(k,J)$. If $d$ and $k$ are odd, the representation has no non-zero invariants.
\end{prop}

The sequences $\#T_{d-1}(k,J)$ and $\#U_{d-1}(k,J)$ can also be described by means of generating functions. By Burnside's lemma, the dimension of the invariant subspace for $d$ even is given by
$$
\#T_{d-1}(k,J)=\frac{1}{d-1}\sum_{r=1}^{d-1}\#\{(a_0,a_1,\ldots,a_{d-2})|a_i=a_{i+r\mbox{ mod }d-1}\}=\frac{1}{d-1}\sum_{r|d-1}\phi(\frac{d-1}{r})\#S_r(\frac{kr}{d-1},J)
$$
where $S_r(k,J)=\emptyset$ if $k$ is not an integer and $\phi$ is Euler's totient function. So the generating function for the sequence $\{\# T_{d-1}(k,J)|k\geq 0\}$ is
\begin{align*}
G_{d-1}(J;T) :&= \sum_{k=0}^\infty\#T_{d-1}(k,J)T^k \\
&= \sum_{k=0}^{\infty}\frac{1}{d-1}T^k\sum_{r|d-1}\phi(\frac{d-1}{r})\#S_r(\frac{kr}{d-1},J) \\
&= \frac{1}{d-1}\sum_{r|d-1}\phi(\frac{d-1}{r})\sum_{\frac{d-1}{r}|k}\#S_r(\frac{kr}{d-1},J)T^k \\
&= \frac{1}{d-1}\sum_{r|d-1}\phi(\frac{d-1}{r})\sum_{s=0}^\infty \#S_r(s,J)T^{\frac{d-1}{r}s} \\
&=\frac{1}{d-1}\sum_{r|d-1}\phi(\frac{d-1}{r})F_r(J;T^{\frac{d-1}{r}})
\end{align*}
Next, suppose that $d$ is odd, and let $(a_0,\ldots,a_{d-2})\in S_{d-1}(k,J)$. Let $r$ be the number of elements in its orbit. Then $\sum_{i=0}^{r-1}a_i=\frac{kr}{d-1}$. We want to count the number of orbits such that this value is even. Since $k=\frac{kr}{d-1}\cdot\frac{d-1}{r}$, if the largest power of $2$ that divides $d-1$ is smaller than the largest power of $2$ dividing $k$, $\frac{kr}{d-1}$ must always be even. Suppose that the largest power of $2$ that divides $k$, $2^{\alpha(k)}$, divides $d-1$. Then $\frac{kr}{d-1}$ is odd if and only if $2^{\alpha(k)}$ divides $\frac{d-1}{r}$, if and only if $r$ divides $\frac{d-1}{2^{\alpha(k)}}$. Therefore $\#U_{d-1}(k,J)=\#T_{d-1}(k,J)$ if $2^{\alpha(k)}$ does not divide $d-1$ and $\#T_{d-1}(k,J)-\#T_{\frac{d-1}{2^{\alpha(k)}}}(\frac{k}{2^{\alpha(k)}},J)$ if it does. The generating function is then
\begin{align*}
\sum_{k=0}^\infty\#U_{d-1}(k,J)T^k &= \sum_{k=0}^\infty\#T_{d-1}(k,J)T^k-\sum_{j\geq 1; 2^j|d-1}\sum_{l\text{ odd}}\#T_{\frac{d-1}{2^j}}(l,J)T^{2^jl} \\
&= G_{d-1}(J;T)-\sum_{j\geq 1;2^j|d-1}H_{\frac{d-1}{2^j}}(J;T^{2^j})
\end{align*}
where
$$
H_{r}(J;T):=\frac{1}{2}(G_{r}(J;T)-G_{r}(J;-T)).
$$

Let $F\in D_\infty^{d-1}\subset D_\infty$ be a geometric Frobenius element, and $w=\sum_{j=0}^{d-2} v_j^{a_0}v_{j+1}^{a_1}\cdots v_{j+d-2}^{a_{d-2}}$ (resp. $w=\sum_{j=0}^{d-2} \beta(h)^{jk}v_j^{a_0}v_{j+1}^{a_1}\cdots v_{j+d-2}^{a_{d-2}}$) a generator of the $I_\infty$-invariant subspace of $\Sym^k V$. $F$ acts on $v_j^{a_0}v_{j+1}^{a_1}\cdots v_{j+d-2}^{a_{d-2}}$ via the character corresponding to $\LL_{\psi(\sum a_i g(\zeta^{j+i} t))}\otimes\LL^{\otimes k}_{\rho(d(d-1)c_d/2)}\otimes{\mathcal Q}^{\otimes k}$ (resp. $\LL_{\psi(\sum a_i g(\zeta^{j+i} t))}\otimes\LL_{\rho(\prod (s_0\zeta^{j+i}t)^{a_i})}\otimes\LL^{\otimes k}_{\rho(d(d-1)c_d/2)}\otimes{\mathcal Q}^{\otimes k}$). Since $\sum a_i g(\zeta^{j+i} t)$ must be a constant polynomial, we have $\LL_{\psi(\sum a_i g(\zeta^{j+i} t))}\cong\LL_{\psi(kb_0)}$. Additionally, if $d$ is odd and $k$ even, $\LL_{\rho(\prod (s_0t)^{a_i})}=\LL_{\rho(s_0t)^k}$ is trivial. We conclude:

\begin{prop} A Frobenius geometric element at infinity acts on the $I_\infty$-invariant subspace of $\Sym^k\Aif$ by multiplication by $\psi(kb_0)\rho(d(d-1)c_d/2)^k g(\psi,\rho)^k$.
\end{prop}

As an immediate consequence we get

\begin{cor}\label{localLfunction} 
The local $L$-function of $\Sym^k\Aif$ at infinity $\det(1-\mathrm{Frob} \> T|(\Sym^k\Aif)^{I_\infty})$ is given by $(1-\psi(kb_0)\rho(d(d-1)c_d/2)^kg(\psi,\rho)^kT)^{\#T_{d-1}(k,J)}$ if $d$ is even, $(1-\psi(kb_0)\rho(d(d-1)c_d/2)^kg(\psi,\rho)^kT)^{\#U_{d-1}(k,J)}$ if $d$ is odd and $k$ is even, and $1$ if $d$ and $k$ are odd.\end{cor}

\begin{thm} The polynomial $M_k(f,T)$ decomposes as a product $P_k(f,T)Q_k(f,T)$, where $Q_k(f,T)$ is given by the formula in Corollary \ref{localLfunction} and $P_k(d,T)$ satisfies a functional equation
$$
{P(T)}=cT^r\overline{P(1/q^{k+1}T)}
$$
where $|c|=q^{r(k+1)/2}$ and $r$ is its degree.
\end{thm}

\begin{proof}
Let $j:\AAA^1\to\PP^1$ be the inclusion. From the exact sequence
$$
0\to \Sym^k\Aif\to j_\star\Sym^k\Aif\to (j_\star\Sym^k\Aif)_\infty\to 0
$$
we get an exact sequence of $\Gal(\FFq/\Fq)$-modules
$$
0\to (j_\star\Sym^k\Aif)^{I_\infty}\to \HH^1_c(\AAA^1,\Sym^k\Aif)\to\HH^1(\PP^1,j_\star\Sym^k\Aif)\to 0
$$
and therefore a decomposition
\begin{align*}
M_k(f,T) &= \det(1-\text{Frob} \> T|\HH^1_c(\AAA^1,\Sym^k\Aif)) \\
&= \det(1-\text{Frob} \> T|(j_\star\Sym^k\Aif)^{I_\infty}) \det(1-\text{Frob} \> T|\HH^1(\PP^1,j_\star\Sym^k\Aif)).
\end{align*}
The first factor is described by the previous corollary. On the other hand, by \cite[Th\'eor\`eme 1.3]{deligne-dualite} we have a perfect pairing
$$
\HH^1(\PP^1,j_\star\Sym^k\Aif)\times\HH^1(\PP^1,j_\star\Sym^k\widehat{\Aif})\to\QQ(-k-1)
$$
where $\widehat{\Aif}$ is the dual of $\Aif$, which is constructed in the same way as $\Aif$ using the complex conjugate character $\bar\psi$ instead of $\psi$. If the eigenvalues of the action of Frobenius on $\HH^1(\PP^1,j_\star\Sym^k\Aif)$ are $\alpha_1,\cdots,\alpha_r$, so that $P_k(f,T)=\prod(1-\alpha_iT)$, it follows that $\overline{P_k(f,T)}=\prod(1-(q^{k+1}/\alpha_i)T)$ and therefore the functional equation holds. Applying the functional equation twice we get $|c|=q^{r(k+1)/2}$.
\end{proof}

\section{Some special cases}\label{someexamples}

We will now see how the previous results apply to some special values of $f$. First, consider the case $f(t)=t^d$. In this case the equation $f'(t)+u(t)^{d-1}=0$ gives $u(t)=r_0t$, where $r_0^{d-1}=-d$. Then $v(t)=t/r_0$, and $g(t)=f(v(t))+v(t)t^{d-1}=t^d(1/r_0^d+1/r_0)=\frac{d-1}{dr_0}t^d$. By corollary \ref{degreeM}, we get that the degree of $M_k(f;T)$ is the $k$-th coefficient in the power series expansion of
$$
\frac{1}{d-1}\left(\frac{1}{(1-T)^{d-1}}-dF_{d-1}(\{1\};T)\right)
$$
where
$$
F_{d-1}(\{1\};T)=\frac{1}{q}\sum_{\gamma\in{\Fq}}\prod_{j=0}^{d-2}(1-\psi(\gamma\zeta^{j})T)^{-1}.
$$
Explicitly,
$$
\deg M_k(f,T)=\frac{1}{d-1}\left({{k+d-2}\choose{d-2}}-d\cdot\#S_{d-1}(k,\{1\})\right).
$$

In particular, for $d=3$
$$
F_2(\{1\};T)=\frac{1}{q}\sum_{\gamma\in{\Fq}}(1-\psi(\gamma)T)^{-1}(1-\psi(-\gamma)T)^{-1}=\frac{1}{p}\sum_{m=0}^{p-1}(1-\exp(\frac{2\pi i m}{p})T)^{-1}(1-\exp(\frac{-2\pi i m}{p})T)^{-1}.
$$
It is easily checked that $S_2(k,\{1\}):=\{(a,b)|a+b=k,a\equiv b (\mod p)\}$ has $\lfloor\frac{k}{p}\rfloor+\delta$ elements, where $\delta=0$ (resp. $\delta=1$) if $k-\lfloor\frac{k}{p}\rfloor$ is odd (resp. even). So in this case we get an explicit formula for the degree:
$$
\deg M_k(f(t)=t^3;T)=\frac{1}{2}\left(k+1-3\left(\left\lfloor\frac{k}{p}\right\rfloor +\delta\right)\right)
$$
If $p>k$ this gives $(k+1)/2$ for $k$ odd and $(k-2)/2$ for $k$ even.

Corollary \ref{localLfunction} states for $f(t)=t^d$ that the local $L$-function of $\Sym^k\Aif$ at infinity is  $(1-\rho(d(d-1)/2)^kg(\psi,\rho)^kT)^{\#T_{d-1}(k,J)}$ if $d$ is even, $(1-\rho(d(d-1)/2)^kg(\psi,\rho)^kT)^{\#U_{d-1}(k,J)}$ if $d$ is odd and $k$ is even and $1$ if $d$ and $k$ are odd. For $d=3$, we can again provide a more explicit expression.

Since $3$ is odd, the local $L$-function is $1$ for $k$ odd. For $k$ even, we can write $\#S_2(k,\{1\})=\lfloor\frac{k}{p}\rfloor+\delta=2\lfloor\frac{k}{2p}\rfloor+1$. Every orbit of ${\mathbb Z}/2{\mathbb Z}$ acting on $S_2(k,\{1\})$ has two elements except for $\{(k/2,k/2)\}$, so $\#T_2(k,\{1\})=\lfloor\frac{k}{2p}\rfloor+1$. $U_2(k,\{1\})$ contains the orbits such that $rk$ is a multiple of $4$. If $k\equiv 0(\mbox{mod }4)$ this includes all orbits. If $k\equiv 2(\mbox{mod }4)$ the orbit $\{(\frac{k}{2},\frac{k}{2})\}$ must be excluded. So the trivial factor for $k$ even is
$$
\begin{array}{ll}
(1-g(\psi,\rho)^kT)^{\lfloor\frac{k}{2p}\rfloor} & \mbox{for }k\equiv 2(\mbox{mod }4)  \vspace{5pt} \\
(1-g(\psi,\rho)^kT)^{\lfloor\frac{k}{2p}\rfloor+1} & \mbox{for }k\equiv 0(\mbox{mod }4) \end{array}
$$
 In particular, for $p>\frac{k}{2}$ the trivial factor of $M_k(t^3,T)$ is $1$ if $k\equiv 2(\mbox{mod }4)$ and $(1-g(\psi,\rho)^kT)$ if $k\equiv 0(\mbox{mod }4)$.

\bigskip

We will now consider the case where $g(t)=\sum b_i t^i$ has $b_i\neq 0$ for $i=1,\ldots,d-2$. This includes the generic case where all coefficients of $g(t)$ are non-zero as a special case. Suppose first that $b_{d-1}=0$ (or, equivalently, that $c_{d-1}=0$). $S_{d-1}(k,J)$ is the set of all $(a_0,\ldots,a_{d-2})\in\Z^{d-1}_{\geq 0}$ such that $\sum a_i=k$ and $\sum a_i\zeta^{ji}=0$ for all $j=1,\ldots,d-2$. The system of equations $\{\sum_i\zeta^{ij}x_i=0|j=1,\ldots,d-2\}$ has rank $d-2$ (since the $(d-2)\times(d-2)$ minors are Vandermonde determinants) and has $(1,1,\ldots,1)$ as a solution, so all solutions must be of the form $(a,a,\ldots,a)$ modulo $p$ for some $a$. Therefore
\begin{align*}
F_{d-1}(J;T) :&= \sum_{k=0}^\infty \# S_{d-1}(k,J) T^k \\
&= \sum_{r=0}^{p-1}\sum_{a_0,\ldots,a_{d-2}=0}^\infty T^{(r+s_0p)+\cdots +(r+s_{d-2}p)} \\
&= \sum_{r=0}^{p-1} T^{(d-1)r}\sum_{a_0,\ldots,a_{d-2}=0}^{\infty} T^{p(a_0+\cdots+a_{d-2})} \\
&=\frac{1-T^{(d-1)p}}{(1-T^p)^{d-1}(1-T^{d-1})}
\end{align*}

Suppose now that $b_{d-1}\neq 0$ (or, equivalently, that $c_{d-1}\neq 0$). Making the change of variable $\hat f(t)=f(t-\frac{c_{n-1}}{nc_n})$ we eliminate the degree $d-1$ term. Moreover, $\mathrm{Ai}_{\hat f}=\R^1\pi_{t!}\LL_{\psi(f(x-\frac{c_{n-1}}{nc_n})+tx)}=\R^1\pi_{t!}\LL_{\psi(f(x)+t(x+\frac{c_{n-1}}{nc_n}))}=\Aif\otimes\LL_{\psi(\frac{c_{n-1}}{nc_n}t)}$ and thus $\Sym^k\Aif=(\Sym^k{\mathrm Ai}_{\hat f})\otimes\LL_{\psi(-\frac{c_{n-1}}{nc_n}t)}^{\otimes k}$. As a representation of $D_\infty$, we have then $\Aif=[d-1]_\star(\LL_{\psi(\hat g(t))}\otimes\LL_{\rho^d(s_0t)})\otimes\LL_{\rho(d(d-1)c_d/2)}\otimes{\mathcal Q}\otimes\LL_{\psi(-\frac{c_{n-1}}{nc_n}t)}=[d-1]_\star(\LL_{\psi(\hat g(t)-\frac{c_{n-1}}{nc_n}t^{d-1})}\otimes\LL_{\rho^d(s_0t)})\otimes\LL_{\rho(d(d-1)c_d/2)}\otimes{\mathcal Q}$. In other words, $g(t)=\hat g(t)-\frac{c_{n-1}}{nc_n}t^{d-1}$.

If $p$ divides $k$, the condition $\sum_i a_i\zeta^{ij}$ for $j=d-1$ is void, so both the dimension of $M_k(f;T)$ and the trivial factor at infinity behave as in the $b_{d-1}=0$ case. If $p$ does not divide $k$, the condition $\sum_i a_i\zeta^{ij}$ does never hold for $j=d-1$, so $S_{d-1}(k,J_{\geq j})=\emptyset$ for $j=1,\ldots,d-1$. In particular, the trivial factor of $M_k(f;T)$ is $1$. Furthermore, applying the formula for the degree, we get
$$
\deg M_k(f,T)=\frac{1}{d-1}\left({{k+d-2}\choose{d-2}}-\#S_{d-1}(k,\{1\})\right).
$$

\bigskip

As a final example, suppose that $d-1$ is prime and $p$ is a multiplicative generator of ${\mathbb F}_{d-1}$. In this case, all non-trivial $(d-1)$-th roots of unity are conjugate over $\Fp$, so $a_0+a_1\zeta+\cdots+a_{d-2}\zeta^{d-2}=0$ if and only if $a_0+a_1\zeta^j+\cdots+a_{d-2}\zeta^{(d-2)j}=0$ for any $j=1,2,\ldots,d-2$. Therefore $S_{d-1}(k,\{1\})=S_{d-1}(k,J)$ for every $J\subset{\mathbb Z}$ such that $J\cap (d-1){\mathbb Z}=\emptyset$. As in the previous example, we conclude that, if $c_{d-1}=0$,
$$
F_{d-1}(J_{\geq j};T)=\frac{1-T^{(d-1)p}}{(1-T^p)^{d-1}(1-T^{d-1})}
$$
for \emph{every} $j\in J$. By corollary \ref{degreeM}, the degree of $M_k(f;T)$ is the $k$-th coefficient of the power series expansion of
$$
\frac{1}{(d-1)(1-T)^{d-1}}-\frac{1}{d-1}\cdot \frac{1-T^{(d-1)p}}{(1-T^p)^{d-1}(1-T^{d-1})}\sum_{j\in J} h(j)=\frac{1}{(d-1)(1-T)^{d-1}}-\frac{d}{d-1}\cdot \frac{1-T^{(d-1)p}}{(1-T^p)^{d-1}(1-T^{d-1})}.
$$

If $c_{d-1}\neq 0$ we have, as in the previous example, the same formula for the degree if $k$ is a multiple of $p$, and the $k$-th coefficient in the power series expansion of 
$$
\frac{1}{(d-1)(1-T)^{d-1}}-\frac{1}{d-1}\cdot \frac{1-T^{(d-1)p}}{(1-T^p)^{d-1}(1-T^{d-1})}
$$
if $k$ is prime to $p$.

\bibliographystyle{amsplain}
\bibliography{References,bibliografia}

\end{document}